\documentclass[a4paper]{amsart}

\usepackage{dsfont,hyperref}

\newtheorem{theorem}{Theorem}
\newtheorem{lemma}[theorem]{Lemma}
\newtheorem{proposition}[theorem]{Proposition}

\theoremstyle{definition}
\newtheorem{remark}[theorem]{Remark}

\DeclareMathOperator{\End}{End}
\DeclareMathOperator{\Hom}{Hom}
\renewcommand{\Im}{\operatorname{Im}}
\DeclareMathOperator{\Ker}{Ker}
\DeclareMathOperator{\Mod}{Mod}

\newcommand{\1}{\mathds{1}}
\newcommand{\C}{\mathcal{C}}
\newcommand{\modlf}{\operatorname{mod}_{lf}}
\renewcommand{\S}{\mathcal{S}}
\newcommand{\set}[1]{\left\{#1\right\}}

\newcommand{\op}{\mathrm{op}}
\newcommand{\T}{\mathcal{T}}

\newcommand{\Z}{\mathbb{Z}}

\title[Decompositions of locally finite endo-length modules]
  {Decompositions of locally finite endo-length modules
  over a skeletally small category}
\author{Pengjie Jiao}
\address{
  College of Sciences,
  China Jiliang University,
  Hangzhou 310018, PR China}
\email{jiaopjie@mail.ustc.edu.cn}

\subjclass[2010]{16D70}
\keywords{locally finite endo-length module,
  indecomposable decomposition}
\date{\today}

\begin{document}

\begin{abstract}
  Given a skeletally small category $\mathcal{C}$, we show that any locally finite endo-length $\mathcal{C}$-module is the direct sum of indecomposable $\mathcal{C}$-modules, whose endomorphism algebra is local.
\end{abstract}

\maketitle

\section{Introduction}

Recently, the category of locally finite dimensional representations of a strongly locally finite quiver over a field is studied in \cite{BautistaLiuPaquette2013Representation}.
But this category needs not be Hom-finite or Krull-Schmidt in general. Hence, given an object, the existence of indecomposable decompositions needs to be shown.

Analogous to the above case, the category of locally finitely generated modules over a locally bounded ring $R$ with enough idempotents is studied in \cite{AngeleriHugeldelaPena2009Locally}.
It is proved that each locally finitely generated $R$-module admits an indecomposable decomposition; see \cite[Main Theorem]{AngeleriHugeldelaPena2009Locally}.

More generally, let $k$ be a commutative ring and $\C$ be a skeletally small category.
We find that we can consider the indecomposable decompositions of locally finite endo-length $\C$-modules.
Here, a $\C$-module $M$ is called locally finite endo-length if $M(a)$ is a finite length $k \End_\C (a)$-module for any $a \in \C$.

Indeed, we have the following decomposition theorem, which is a generalization of
\cite[Main Theorem(a)]{AngeleriHugeldelaPena2009Locally}.

\begin{theorem}\label{thm:dec}
  Let $\C$ be a skeletally small category, and $M$ be a locally finite endo-length $\C$-module. Then there exists a decomposition $M = \bigoplus_{i \in \Lambda} M_i$ such that each $M_i$ is indecomposable and $\End_{\Mod \C} (M_i)$ is local.
\end{theorem}

\section{Proof of the main result}

Let $k$ be a commutative ring and $\C$ be a skeletally small category. We denote by $\Mod k$ the category of $k$-modules.

Recall that a $\C$-module $M$ means a covariant functor $M \colon \C \to \Mod k$.
We mention that $M$ will be required to be $k$-linear if $\C$ is.
A morphism $f \colon M \to N$ of $\C$-modules means a natural transformation. In other words, it consists of a collection of morphisms $f_a \colon M(a) \to N(a)$ for any $a \in \C$, such that $N(\alpha) \circ f_a = f_b \circ M(\alpha)$ for any morphism $\alpha \colon a \to b$ in $\C$.
We denote by $\Mod \C$ the category of $\C$-modules.

Let $M$ be a $\C$-module.
For each $a \in \C$, we have the $k$-algebra morphism
$k \End_\C (a) \to \End_k (M(a))$ induced by $M$.
We mention that $k \End_\C (a)$ should be replaced by $\End_\C (a)$ if $\C$ is $k$-linear.
Then $M(a)$ becomes a $k \End_\C (a)$-module. Given a morphism $f \colon M \to N$, the $k$-linear map $f_a \colon M(a) \to N(a)$ is a morphism of $k \End_\C (a)$-modules.

Recall that a $\C$-submodule $N$ of $M$ is a $\C$-module such that $N(a) \subseteq M(a)$ for any $a \in \C$ and $N(\alpha)$ is the restriction of $M(\alpha)$ for any morphism $\alpha$ in $\C$. The factor $\C$-module $M/N$ means a $\C$-module such that $(M/N) (a) = M(a)/N(a)$ for any $a \in \C$ and $(M/N) (\alpha)$ is induced by $M(\alpha)$.

Given a morphism $f \colon M \to N$ in $\Mod \C$, the kernel $\Ker f$ and image $\Im f$ are defined pointwisely. Given a set $\set{M_i \middle\vert i \in I}$ of $\C$-modules, the direct product $\prod_{i \in I} M_i$ and direct sum $\bigoplus_{i \in I} M_i$ are also defined pointwisely.
It is well known that $\Mod \C$ is an abelian category admitting direct products and direct sums.

Let $\Omega$ be a chain of submodules of some $\C$-module $M$. We mention that $\bigcup_{X \in \Omega} X$ and $\bigcap_{X \in \Omega} X$ are submodules of $M$.

We call a $\C$-module $M$ \emph{locally finite endo-length} if $M(a)$ is of finite length as $k \End_\C (a)$-module for any $a \in \C$. We denote by $\modlf \C$ the full subcategory of $\Mod \C$ formed by locally finite endo-length $\C$-modules. We observe that $\modlf \C$ is an abelian subcategory of $\Mod \C$, which is closed under extensions.

We mention the following observation.

\begin{lemma}\label{lem:chain}
  Let $M$ be a locally finite endo-length $\C$-module and $\Omega$ be a chain of submodules of $M$.
  \begin{enumerate}
    \item
      For any $a \in \C$, there exists some $N \in \Omega$ such that
      $N (a) = \left( \bigcup_{X \in \Omega} X \right) (a)$.
    \item
      For any $a \in \C$, there exists some $N \in \Omega$ such that
      $N (a) = \left( \bigcap_{X \in \Omega} X \right) (a)$.
  \end{enumerate}
\end{lemma}

\begin{proof}
  We only prove (1). Let $a \in \C$. Then $M(a)$ is a finite length $k \End_\C (a)$-module. We observe that $X (a)$ is a $k \End_\C (a)$-submodule of $M(a)$ for any $X \in \Omega$. Therefore there exists some $N \in \Omega$ such that $N (a) = X (a)$ for any $X \supseteq N$ in $\Omega$. Then the result follows.
\end{proof}

\begin{lemma}\label{lem:dec}
  Let $M$ be a locally finite endo-length $\C$-module. Assume $\set{X_i \middle\vert i \in I}$ and $\set{Y_i \middle\vert i \in I}$ are chains of submodules of $M$ for some index set $I$, such that $M = X_i \oplus Y_i$ for any $i \in I$. Then $M = \left( \bigcup_{i \in I} X_i \right) \oplus \left( \bigcap_{i \in I} Y_i \right)$.
\end{lemma}

\begin{proof}
  Given $i, j \in I$, we let $i \leq j$ if $X_i \subseteq X_j$. Then $I$ becomes a totally ordered set.
  We observe that $X_i \subseteq X_j$ if and only if $Y_i \supseteq Y_j$, since $X_i \oplus Y_i = M = X_j \oplus Y_j$.

  For each $a \in \C$, Lemma~\ref{lem:chain} implies that there exist some $j_1, j_2 \in I$ such that $X_{j_1} (a) = \left( \bigcup_{i \in I} X_i \right) (a)$ and $Y_{j_2} (a) = \left( \bigcap_{i \in I} Y_i \right) (a)$. We set $j = \max \set{j_1, j_2}$.
  Then we have that $X_j (a) = X_{j_1} (a) = \left( \bigcup_{i \in I} X_i \right) (a)$ and $Y_j (a) = Y_{j_2} (a) = \left( \bigcap_{i \in I} Y_i \right) (a)$.
  Since $M(a) = X_j (a) \oplus Y_j (a)$, we have that
  $M(a) = \left( \bigcup_{i \in I} X_i \right) (a) \oplus \left( \bigcap_{i \in I} Y_i \right) (a)$.
  It follows that
  $M = \left( \bigcup_{i \in I} X_i \right) \oplus \left( \bigcap_{i \in I} Y_i \right)$.
\end{proof}

We first show that each nonzero locally finite endo-length $\C$-module admits an indecomposable direct summand.

\begin{lemma}\label{lem:indec.summand}
  Let $M$ be a locally finite endo-length $\C$-module. For any $a \in \C$ with $M(a) \neq 0$, there exists an indecomposable direct summand $N$ of $M$ such that $N(a) \neq 0$.
\end{lemma}

\begin{proof}
  Let $\S$ be the set of pairs $(X, Y)$, where $X$ and $Y$ are submodules of $M$ with $M = X \oplus Y$ and $Y(a) \neq 0$.
  We observe that $\S$ is nonempty since $(0, M) \in \S$.
  We let $(X, Y) \leq (X', Y')$ if $X \subseteq X'$ and $Y \supseteq Y'$. Then $\S$ becomes a poset.

  Let $\T = \set{(X_i, Y_i) \middle\vert i \in I}$ be a chain in $\S$, for some index set $I$.
  Then $\set{X_i \middle\vert i \in I}$ and $\set{Y_i \middle\vert i \in I}$ are chains of submodules of $M$.
  We set $X = \bigcup_{i \in I} X_i$ and $Y = \bigcap_{i \in I} Y_i$.
  It follows form Lemma~\ref{lem:dec} that $M = X \oplus Y$.
  We observe by Lemma~\ref{lem:chain} that $Y(a) \neq 0$. Then $(X, Y)$ is an upper bound of $\T$ in $\S$.

  By Zorn's lemma, there exists a maximal element $(L, N)$ in $\S$.
  We claim that $N$ is indecomposable. Indeed, assume submodules $N'$ and $N''$ of $N$ satisfy $N = N' \oplus N''$. We observe that at least one of $N'(a)$ and $N''(a)$ is nonzero, since $N(a) \neq 0$. We may assume $N''(a) \neq 0$. Then $(L \oplus N', N'')$ is an element greater than or equal to $(L, N)$ in $\S$. Since $(L, N)$ is maximal in $\S$, we have that $(L \oplus N', N'') = (L, N)$. In particular, $N' = 0$ and $N'' = N$.
  It follows that $N$ is an indecomposable direct summand of $M$ with $N(a) \neq 0$.
\end{proof}

Then we show the existence of indecomposable decomposition for a locally finite endo-length $\C$-module.

\begin{proposition}\label{prop:dec}
  Each locally finite endo-length $\C$-module $M$ admits a decomposition $M = \bigoplus_{i \in \Lambda} M_i$ such that each $M_i$ is indecomposable.
\end{proposition}

\begin{proof}
  Denote by $\Omega$ the set of indecomposable direct summands of $M$. Let $\S$ be the set of pairs $(I, Y)$, where $I$ is a subset of $\Omega$ and $Y$ is a submodule of $M$ with $M = \left( \bigoplus_{X \in I} X \right) \oplus Y$.
  We observe that $\S$ is nonempty since $(\emptyset, M) \in \S$.
  We let $(I, Y) \leq (I', Y')$ if $I \subseteq I'$ and $Y \supseteq Y'$. Then $\S$ becomes a poset.

  Let $\T = \set{(I_j, Y_j) \middle\vert j \in J}$ be a chain in $\S$, for some index set $J$.
  We observe that $\set{\bigoplus_{X \in I_j} X \middle\vert j \in J}$ and $\set{Y_j \middle\vert j \in J}$ are chains of submodules of $M$.
  We set $I_* = \bigcup_{j \in J} I_j$ and $Y_* = \bigcap_{j \in J} Y_j$.
  We observe that $\bigoplus_{X \in I_*} X = \bigcup_{j \in J} \left( \bigoplus_{X \in I_j} X \right)$.
  It follows form Lemma~\ref{lem:dec} that $M = \left( \bigoplus_{X \in I_*} X \right) \oplus Y_*$.
  Then $(I_*, Y_*)$ is an upper bound of $\T$ in $\S$.

  By Zorn's lemma, there exists a maximal element $(I, Y)$ in $\S$.
  We claim that $Y = 0$.
  Otherwise, Lemma~\ref{lem:indec.summand} implies that there exist some indecomposable direct summand $Y'$ of $Y$. Assume $Y = Y' \oplus Y''$ for some submodule $Y''$ of $Y$. We observe that $Y'$ is an indecomposable direct summand of $M$. Then $(I \cup \set{Y'}, Y'') $ is an element greater than $(I, Y)$ in $\S$, which is a contradiction.

  Hence $Y = 0$ and then $M = \bigoplus_{X \in I} X$.
  We observe that each $X \in I$ is an indecomposable direct summand of $M$.
  Then the result follows.
\end{proof}

Let $M$ be a $\C$-module. For any endomorphism $f \colon M \to M$, we have two chains $\set{\Ker f^n \middle\vert n \geq 0}$ and $\set{\Im f^n \middle\vert n \geq 0}$ of submodules of $M$. We set $\Ker f^\infty = \bigcup_{n \geq 0} \Ker f^n$ and $\Im f^\infty = \bigcap_{n \geq 0} \Im f^n$.

The following result is essentially contained in
\cite[Section~3.6 and Remark~3.8.5]{GabrielRoiter1992Representations};
compare \cite[Corollary~15]{AngeleriHugeldelaPena2009Locally}.

\begin{proposition}\label{prop:indec}
  The endomorphism algebra of an indecomposable locally finite endo-length $\C$-module is local.
\end{proposition}

\begin{proof}
  Let $M$ be an indecomposable locally finite endo-length $\C$-module and let $f \colon M \to M$ be an endomorphism.
  For any $a \in \C$, we have that $M(a)$ is a finite length $k \End_\C (a)$-module. For any $\C$-submodule $N$ of $M$, we observe that $N(a)$ is a $k \End_\C (a)$-submodule of $M(a)$.
  Then Fitting's lemma implies that there exists some positive integer $n_a$ such that
  \[
    M(a) = \Ker f^{n_a} (a) \oplus \Im f^{n_a} (a)
    = \Ker f^\infty (a) \oplus \Im f^\infty (a).
  \]
  It follows that $M = \Ker f^\infty \oplus \Im f^\infty$.
  Since $M$ is indecomposable, then either $\Ker f^\infty$ or $\Im f^\infty$ is zero.

  If $\Ker f^\infty = 0$, then $f_a$ is a bijection for any $a \in \C$, which makes $f$ an isomorphism.
  If $\Im f^\infty = 0$, then $f_a$ is nilpotent for any $a \in \C$. Then $(\1_M-f)_a$ is a bijection with the inverse $\sum_{i=0}^{n_a} f_a^i$. This makes $\1_M-f$ an isomorphism.
  That is to say, either $f$ or $\1_M-f$ is an isomorphism.
  Then the result follows.
\end{proof}

Now, we can complete the proof of Theorem~\ref{thm:dec}.
We observe that Proposition~\ref{prop:dec} gives a decomposition $M = \bigoplus_{i \in \Lambda} M_i$ such that each $M_i$ is indecomposable.
Proposition~\ref{prop:indec} implies that each $\End_{\Mod \C} (M_i)$ is local.
Then Theorem~\ref{thm:dec} follows.

\begin{remark}
  \begin{enumerate}
    \item
      We call a $\C$-module $M$ \emph{locally finite length} if $M(a)$ is of finite length as $k$-module for any $a \in \C$. Since $k \End_\C (a)$ is a $k$-algebra, then $l_k M(a) \geq l_{k \End_\C (a)} M(a)$. Hence $M$ lies in $\modlf \C$.

      Specially, if $k$ is a field and $\C$ is the categorification of a quiver $Q$, then $M$ is just a locally finite dimensional representation of $Q$.
      In this case, Theorem~\ref{thm:dec} for $M$ may be known to experts.
    \item
      Let $R$ be a ring with $R = \bigoplus_{a \in \Lambda} e_a R = \bigoplus_{a \in \Lambda} R  e_a$, where $\set{e_a \middle\vert a \in \Lambda}$ is a family of pairwise orthogonal idempotents. Then $R$ is a $\Z$-algebra (needs not contain identity).
      Assume $R$ is locally bounded (i.e., $e_a R$ and $R e_a$ are of finite $R$-length for each $a \in \Lambda$).

      We construct a category $\C$ as follows. We choose $\Lambda$ as the set of objects. For a pair of objects $a$ and $b$, we set $\Hom_\C (a, b) = e_b R e_a$. The composition of $r_1 \colon a \to b$ and $r_2 \colon b \to c$ is given by $r_2 \circ r_1 = r_2 r_1$. Then $\C$ forms a $\Z$-linear category and $\1_a = e_a$ for any object $a$ in $\C$.

      We mention that $\End_\C (a) = e_a R e_a$ is an artinian ring for any object $a$ in $\C$;
      see \cite[Proposition~3]{AngeleriHugeldelaPena2009Locally}.
      Then a locally finitely generated unitary right $R$-module $M$ in the sense of
      \cite{AngeleriHugeldelaPena2009Locally} is a locally finite endo-length $\C^\op$-module.
      In this case, Theorem~\ref{thm:dec} for $M$ coincides with
      \cite[Main Theorem(a)]{AngeleriHugeldelaPena2009Locally}.
  \end{enumerate}
\end{remark}

\section*{Acknowledgements}

The author thanks Professor~Yu Ye for the direction on the proof and thanks Professor~Xiao-Wu Chen for some suggestions.


\begin{thebibliography}{1}

\bibitem{AngeleriHugeldelaPena2009Locally}
{L.~Angeleri~H{\"u}gel and J.~A. de~la Pe{\~n}a}, \emph{Locally finitely
  generated modules over rings with enough idempotents}, J. Algebra Appl.,
  8 (2009), pp.~885--901.

\bibitem{BautistaLiuPaquette2013Representation}
{R.~Bautista, S.~Liu, and C.~Paquette}, \emph{Representation theory of
  strongly locally finite quivers}, Proc. Lond. Math. Soc. (3), 106 (2013),
  pp.~97--162.

\bibitem{GabrielRoiter1992Representations}
{P.~Gabriel and A.~V. Roiter}, \emph{Representations of finite-dimensional
  algebras}, in Algebra {VIII}, vol.~73 of Encyclopaedia Mathematical Sciences,
  Springer, Berlin, 1992, pp.~1--177.

\end{thebibliography}

\end{document}